\documentclass[12pt,a4paper]{article}
\usepackage{amsmath,amssymb,amsthm,graphicx,color}
\usepackage[left=1in,right=1in,top=1in,bottom=1in]{geometry}
\usepackage{cite,multirow}
\usepackage[british]{babel}
\usepackage{enumerate}
\usepackage{hyperref} 

\hypersetup{
    colorlinks=false,
    pdfborder={0 0 0},
}

\newtheorem{theorem}{Theorem}[section]
\newtheorem{corollary}[theorem]{Corollary}
\newtheorem{proposition}[theorem]{Proposition}
\newtheorem{lemma}[theorem]{Lemma}
 
\theoremstyle{remark}
\newtheorem*{remark}{Remark}
\newtheorem*{remarks}{Remarks}

\allowdisplaybreaks

\newcommand{\R}{\mathbb{R}}
\newcommand{\RP}{\mathbb{R}_+}
\newcommand{\N}{\mathbb{N}}
\newcommand{\Sp}{\mathbb{S}}

\newcommand{\ud}{\textup{d}}
\newcommand{\re}{{\mathrm{e}}}
\newcommand{\eps}{\varepsilon}

\renewcommand{\Pr}{\mathbb{P}}
\DeclareMathOperator{\Exp}{\mathbb{E}}
\DeclareMathOperator{\Var}{\mathbb{V}ar}
\DeclareMathOperator{\hull}{hull} 
 
\DeclareMathOperator{\Int}{int} 
\newcommand{\as}{{\ \mathrm{a.s.}}}

\DeclareMathOperator{\trace}{tr} 
\newcommand{\tra}{{\scalebox{0.6}{$\top$}}}
\newcommand{\per}{{\mkern -1mu \scalebox{0.5}{$\perp$}}}

\newcommand{\cH}{{\mathcal{H}}}
\newcommand{\cA}{{\mathcal{A}}}
\newcommand{\cN}{{\mathcal{N}}}
\newcommand{\cS}{{\mathcal{S}}}
\newcommand{\cL}{{\mathcal{L}}}
\newcommand{\cC}{{\mathcal{C}}}
\newcommand{\cK}{{\mathcal{K}}}

\newcommand{\1}{{\bf 1}}
\newcommand{\tod}{\stackrel{d}{\longrightarrow}}
\newcommand{\leb}{\cA}
\newcommand{\sperp}{\sigma^2_{\mu_\per}}
\newcommand{\spara}{\sigma^2_{\mu}}
\newcommand{\blob}{\mkern 1.5mu \raisebox{1.7pt}{\scalebox{0.4}{$\bullet$}} \mkern 1.5mu}

\makeatletter
\def\namedlabel#1#2{\begingroup  
    (#2)%
    \def\@currentlabel{#2}%
    \phantomsection\label{#1}\endgroup
}
\makeatother

\begin{document}

\title{Convex hulls of random walks and their scaling limits}
\author{Andrew R.\ Wade \and Chang Xu}
 
\maketitle

\begin{abstract}
For the perimeter length and the area of the
convex hull of the first $n$ steps of a planar random walk,
we study $n \to \infty$ mean and variance asymptotics and establish non-Gaussian
distributional limits. Our results apply to random walks with drift
(for the area) and walks with no drift (for both area and perimeter length) under mild moments assumptions on the increments. These results complement
and contrast with
previous work which showed that the perimeter length in the case with drift satisfies a central limit theorem.
We deduce these results from weak convergence statements for the convex hulls of random walks
to scaling limits defined in terms of convex hulls of certain Brownian motions. We give bounds
that confirm that the   limiting variances in our results are non-zero. 
\end{abstract}

\medskip


\medskip


\section{Introduction}

Random walks are classical objects in probability theory. Recent
 attention has focussed on various geometrical aspects of random walk trajectories.
Many of the questions of stochastic geometry, traditionally concerned with functionals of independent random points, are also of interest for point sets generated by random walks. 
Here we examine
 the asymptotic behaviour of the \emph{convex hull}
 of the first $n$ steps of a random walk in $\R^2$, a
  natural geometrical characteristic of the process.
Study of the convex hull of planar random walk goes back to Spitzer and Widom \cite{sw}
 and the continuum analogue, convex hull of planar Brownian motion,
to L\'evy \cite[\S52.6, pp.~254--256]{levybook}; both have received renewed interest recently, in part motivated by
applications arising
for example in modelling the `home range' of animals. See \cite{mcr} for a recent survey
of motivation and previous work. 
The method of the present paper in part relies on an analysis of \emph{scaling limits},
and thus links
 the discrete and continuum settings.

Let $Z$ be a random vector in $\R^2$,  and let $Z_1, Z_2, \ldots$ be independent copies of $Z$.
Set $S_0 := 0$ and $S_n := \sum_{k=1}^n Z_k$; $S_n$ is the planar random walk, started at the origin, 
with increments distributed as $Z$.
We will impose a moments condition of the following form:
\begin{description}
\item[\namedlabel{ass:moments}{M$_p$}]
Suppose that $\Exp [ \| Z \|^p ] < \infty$.
\end{description}
\emph{Throughout the paper}
 we assume (usually tacitly) 
that the $p=2$ case of \eqref{ass:moments} holds.
For several of our results we impose a stronger condition and
assume that \eqref{ass:moments} holds for some $p>2$, in which case we
say so explicitly.

Given \eqref{ass:moments} holds for some $p \geq 2$, 
 both
 $\mu := \Exp Z \in \R^2$, the mean drift vector
of the walk, and
$\Sigma := \Exp [ (Z - \mu)(Z-\mu)^\tra]$, the covariance
matrix associated with $Z$, are well defined;
 $\Sigma$ is positive semidefinite and symmetric.
We also write $\sigma^2 := \trace \Sigma = \Exp [ \| Z - \mu \|^2 ]$.  Here and elsewhere   $Z$ and $\mu$ are viewed as column vectors, and $\| \blob \|$ is the Euclidean norm. 

For a subset $\cS$ of $\R^d$, its convex hull, which we denote
$\hull \cS$, is   the smallest  
  convex set that contains $\cS$.
We are interested in $\hull \{ S_0, S_1, \ldots, S_n \}$,
which 
is a (random) convex polygon,
and in particular in its
 perimeter length $L_n$ and area $A_n$. (See Figure \ref{fig1}.)

\begin{figure}
\center
\includegraphics[width=0.6\textwidth]{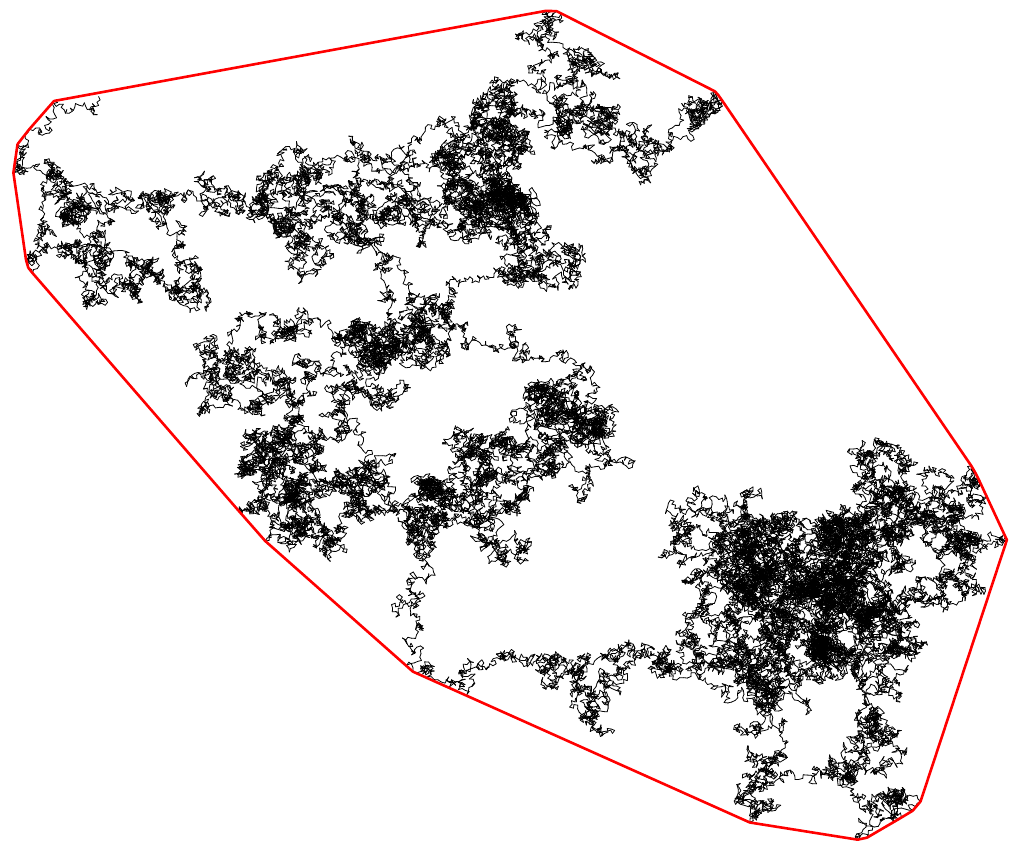}
\caption{Simulated path of a zero-drift random walk and its convex hull.}
\label{fig1}
\end{figure} 

The perimeter length $L_n$ has received some attention in the literature, initiated by
the remarkable formula of Spitzer and Widom \cite{sw}, which states that
\begin{equation}
\label{eq:sw}
\Exp L_n =  2 \sum_{k=1}^n k^{-1} \Exp \| S_k \|,  \text{ for all } n \in \N := \{ 1 ,2, \ldots \}.
\end{equation}
Much later, Snyder and Steele \cite{ss} obtained the law of large numbers 
$\lim_{n \to \infty} n^{-1} L_n = 2 \| \mu \|$, a.s.;
this is stated for the case $\mu \neq 0$ in \cite{ss} but the proof works equally well in the case $\mu=0$.
To prove their law of large numbers, Snyder and Steele used the Spitzer--Widom formula \eqref{eq:sw} and the variance bound 
\cite[Theorem 2.3]{ss}
\begin{equation}
\label{eq:ss}
n^{-1} \Var L_n \leq \frac{\pi^2  \sigma^2}{2},  \text{ for all } n \in \N.\end{equation}
 The natural question of the second-order behaviour of $L_n$  was left largely open;
  similar questions may be posed about $A_n$.
 
In \cite{wx}  a martingale-difference analysis was used to show that 
\begin{equation}
\label{eq:wx}
 \text{if } \mu \neq 0: ~~ \lim_{n \to \infty} n^{-1} \Var L_n = 4 \spara ,\end{equation}
where we introduce the decomposition $\sigma^2 = \spara + \sperp$ with 
\[ \spara := \Exp \left[ \left( ( Z - \mu) \cdot \hat \mu \right)^2 \right] = \Exp [ ( Z \cdot \hat \mu )^2 ] - \| \mu \|^2 \in \RP.\]
Here and elsewhere,  `$\cdot$' denotes the scalar product,  $\hat \mu := \| \mu \|^{-1} \mu$ for $\mu \neq 0$,
and $\RP := [0,\infty)$. 
In \cite{wx}, a central limit theorem to accompany \eqref{eq:wx} was also obtained: provided $\spara > 0$,
$n^{-1/2} ( L_n - \Exp L_n )$ converges in distribution to
a normal random variable with mean 0 and variance $4\spara$.
If $\Sigma$ is positive definite, then both $\spara$ and $\sperp$ are  strictly positive,
but our results are still of interest when one or other of them is zero (the case where both are zero
being entirely trivial).

The aims of the present
 paper are to provide second-order information for $L_n$ in the case $\mu =0$, and   to study the area $A_n$
for both the cases $\mu =0$ and $\mu \neq 0$. For example,  we will show that
\begin{alignat}{2}
\label{eq:three_vars}
\text{if } \mu \neq 0: ~~ &   && \lim_{n \to \infty} n^{-3} \Var A_n   = v_+ \| \mu \|^2 \sperp ; \nonumber\\
\text{if } \mu = 0: ~~  & \lim_{n \to \infty} n^{-1} \Var L_n = u_0 ( \Sigma ) , \text{ and } &&  \lim_{n \to \infty} n^{-2} \Var A_n   = v_0 \det \Sigma  
 .\end{alignat}
The quantities $ v_0$ and $v_+$ in \eqref{eq:three_vars} are finite and positive, 
as is $u_0( \blob )$ provided $\sigma^2 \in (0,\infty)$,
and these quantities are in fact variances associated with convex hulls of Brownian scaling limits
for the walk. These scaling limits provide the basis of the analysis in this paper; the methods are necessarily quite different
from those in \cite{wx}. The result $\lim_{n \to \infty} n^{-1} \Var L_n > 0$ in the case $\mu =0$ answers a question raised by Snyder and Steele \cite[\S5]{ss}.
For the constants $u_0(I)$ ($I$ being the identity matrix), $v_0$, and $v_+$, Table \ref{table2} gives 
numerical evaluations of rigorous bounds that we prove in Proposition \ref{prop:var_bounds} below, plus estimates from simulations. 

\begin{table}[!h]
\center
\def\arraystretch{1.4}
\begin{tabular}{c|ccc}
        &   lower bound   & simulation estimate  & upper bound \\
\hline
  $u_0 ( I)$ &  $2.65 \times 10^{-3}$  &  1.08   &  9.87   \\
  $v_0$      &  $8.15  \times 10^{-7}$ &  0.30   &  5.22   \\
  $v_+$      &  $1.44  \times 10^{-6}$ &  0.019  &  2.08   
	\end{tabular}
\caption{Each of the simulation estimates is
 based on $10^5$ instances of a walk of length $n = 10^5$. The final digit in each of the numerical upper (lower)
bounds has been rounded up (down).}
\label{table2}
\end{table}

Furthermore, we show below that  distributional
limits accompanying the three variance asymptotics in \eqref{eq:three_vars} are \emph{non-Gaussian}, excluding trivial cases,
 by contrast to the central limit theorem accompanying \eqref{eq:wx} from \cite{wx}.
Also notable is the comparison between the variance asymptotics for $\mu \neq 0$ in \eqref{eq:wx} and \eqref{eq:three_vars}:
each of the components $\spara$ and $\sperp$ of $\sigma^2$  contributes
 to exactly one of the asymptotics for $\Var L_n$ and $\Var A_n$.
Other results that we present below include asymptotics for expectations. 

\paragraph{Examples.}{Here are some examples to illustrate a range of
asymptotic behaviours exhibited some very simple random walks. 
We summarize what now is known in general in Table \ref{table1}.
\begin{itemize}
\item Suppose that $Z$ takes Cartesian vector values $(1,1)$, $(-1,-1)$, $(-1,1)$ and $(1,-1)$, each with probability $1/4$.
Then $S_n$ is symmetric simple random walk on $\mathbb{Z}^2$ with
 $\mu = (0,0)$ and $\sigma^2 = 2$.
We show below that $n^{-1/2} \Exp L_n \to \sqrt{8 \pi}$ (see also \cite{sw})
and $n^{-1} \Var L_n \to u_0(I) \in (0,\infty)$, while  $n^{-1} \Exp A_n \to \frac{\pi}{2}$ (see also \cite{bnb})
and $n^{-2} \Var A_n \to v_0 \in (0,\infty)$. 
\item Suppose $Z$ takes values $(1,1)$ and $(1,-1)$, each with probability $1/2$.
Then $S_n$ can be viewed as the space-time diagram 
of \emph{one-dimensional}
simple symmetric random walk. 
Here $\mu = (1,0)$, $\spara =0$, and $\sperp =1$.
It is known that $n^{-1} \Exp L_n \to2$ \cite{sw,ss}
and $\Var L_n = o(n)$ \cite{wx};
we show below that $n^{-3/2} \Exp A_n \to \frac{1}{3} \sqrt{2\pi}$
and $n^{-3} \Var A_n \to v_+ \in (0,\infty)$.
\item Suppose $Z$ takes values $(2,0)$ and $(0,0)$, each with probability $1/2$.
Now $\mu = (1,0)$,  $\spara =1$, and $\sperp =0$.
This time  $n^{-1} \Exp L_n \to2$ \cite{sw,ss}
and $n^{-1} \Var L_n \to 4$ \cite{wx};
trivially $A_n =0$ a.s.
\end{itemize}}

\begin{table}[!h]
\center
\def\arraystretch{1.4}
\begin{tabular}{cc|ccc}
 & &  limit exists for $\Exp$ & limit exists for $\Var$ & limit law \\
\hline
\multirow{2}{*}{ $\mu = 0$ } 
 &  $L_n$ & $n^{-1/2} \Exp L_n$$^\mathsection$  & $n^{-1} \Var L_n$  & non-Gaussian  \\
  & $A_n$ & $n^{-1} \Exp A_n$$^{\mathparagraph}$  & $n^{-2} \Var A_n$  & non-Gaussian  \\
\hline 
 \multirow{2}{*}{ $\mu \neq 0$ }  
 &  $L_n$ & $n^{-1} \Exp L_n$$^\mathsection$$^\dagger$  & $n^{-1} \Var L_n$$^\ddagger$ &  Gaussian$^\ddagger$ \\
  & $A_n$ & $n^{-3/2} \Exp A_n$ & $n^{-3} \Var A_n $ & non-Gaussian
\end{tabular}
\caption{Results originate from: $\mathsection$\!\cite{sw}; $\dagger$\!\cite{ss}; $\ddagger$\!\cite{wx};  $\mathparagraph$\!\cite{bnb} (in part);
the rest are new.
The limit laws exclude degenerate cases when associated variances vanish.}
\label{table1}
\end{table}

The outline of the rest of the paper is as follows. In Section \ref{sec:scaling_limits}
we describe  our scaling limit approach, and carry it through after presenting the necessary preliminaries;
the main results of this section, Theorems \ref{thm:limit-zero} and \ref{thm:limit-drift},
give weak convergence statements for convex hulls of random walks in the case of zero and non-zero drift, respectively.
Armed with these weak convergence results, we present 
asymptotics for expectations and variances of the quantities $L_n$ and $A_n$ in Section \ref{sec:asymptotics};
the arguments in this section rely in part on the scaling limit apparatus, and in part on direct random walk computations.
 This section concludes with upper and lower bounds for the limiting variances. Finally, Appendix \ref{sec:rw_norms} 
collects some auxiliary results on random walks that we use.

\section{Scaling limits for convex hulls}
\label{sec:scaling_limits}

\subsection{Overview}
\label{sec:outline}

We describe the general idea of our approach.
Recall that $S_n = \sum_{k=1}^n Z_k$ is the location of our random walk in $\R^2$ after $n$ steps. Write $\cS_n := \{ S_0, S_1, \ldots, S_n \}$.
Our strategy to study properties of the random convex set $\hull \cS_n$ (such as $L_n$ or $A_n$)
is to seek a weak limit for 
a suitable scaling of $\hull \cS_n$, which
we must hope to be 
the convex hull of some scaling limit representing the walk $\cS_n$.

In the case of  zero drift ($\mu = 0$) a candidate scaling limit for the walk is readily identified
 in terms  
 of planar Brownian motion. For the case $\mu \neq 0$,  the `usual' approach of
centering and then scaling the walk (to again obtain planar Brownian motion) is not 
 useful in our context, as this transformation
does not act on the convex hull in any sensible way. A better idea is to scale space differently in the direction of $\mu$ and in the
orthogonal direction.

In other words, in either case we consider
$\phi_n (\cS_n)$ for some \emph{affine} continuous scaling function
$\phi_n : \R^2 \to \R^2$. 
The convex hull is preserved under affine transformations, so 
\[ \phi_n ( \hull \cS_n ) = \hull  \phi_n ( \cS_n )  ,\]
the convex hull of a random set 
which will have a weak limit. We will then be able to deduce
scaling limits for quantities $L_n$ and $A_n$ provided, first, that we work in suitable spaces
on which our functionals of interest 
enjoy continuity, so that we can appeal to the continuous mapping theorem for weak limits,
and, second, that $\phi_n$ acts on length and area by simple scaling. The usual $n^{-1/2}$ scaling
when $\mu =0$ is fine; for $\mu \neq 0$ we scale space in one coordinate by $n^{-1}$ and in the other by $n^{-1/2}$,
which acts nicely on area, but \emph{not}  length. Thus these methods work exactly in the three cases
corresponding to \eqref{eq:three_vars}.

In view of the scaling limits that we expect, it is natural to work not with point
sets like $\cS_n$, but with continuous \emph{paths}; instead of $\cS_n$
we consider the interpolating path constructed as follows.
For each $n \in \N$ and all $t \in [0,1]$, define
\[ X_n (t) :=    S_{\lfloor nt \rfloor} + (nt - \lfloor nt \rfloor ) \left( S_{\lfloor nt \rfloor +1} - S_{\lfloor nt \rfloor} \right)  = S_{\lfloor nt \rfloor} + (nt - \lfloor nt \rfloor ) Z_{\lfloor nt \rfloor +1} .\]
Note that $X_n (0) =  S_0$ and $X_n (1) =   S_n$. 
Given $n$, we are interested in the convex hull of the image in $\R^2$ of the interval
$[0,1]$ under 
the continuous function 
$X_n$. Our scaling limits will be of the same form.

\subsection{Paths, hulls, and hulls of paths}
\label{sec:paths-and-hulls}

We introduce the setting in which we will describe our scaling limit results.
At this point, it is no extra difficulty to work in $\R^d$ for general $d \geq 2$.
 Let $\rho(x,y) = \| x-y\|$ denote the Euclidean distance between $x$ and $y$ in $\R^d$.
For $T > 0$, let $\cC ( [0,T] ; \R^d )$ denote the class of continuous functions
from $[0,T]$ to $\R^d$. Endow $\cC ( [0,T] ; \R^d )$ with the supremum metric
\[ \rho_\infty ( f, g) := \sup_{t \in [0,T]} \rho  ( f(t), g(t) ) , ~\text{for } f,g \in \cC ( [0,T] ; \R^d ). \]
Let $\cC^0 ( [0,T] ; \R^d )$ denote those functions in $\cC ( [0,T] ; \R^d )$  that map $0$ to the origin in $\R^d$.

Usually, we work with $T=1$, in which case we write simply
\[ \cC_d := \cC ( [0,1] ; \R^d ) , ~~\text{and}~~ \cC_d^0 := \{ f \in \cC_d : f(0) = 0 \} .\]
For example, $X_n \in \cC_d^0$ for each $n$.
For $f \in \cC ( [0,T] ; \R^d )$ and $t \in [0,T]$, define $f [0,t] := \{ f(s) : s \in [0,t] \}$, the image of $[0,t]$ under $f$.  Since $[0,t]$ is compact and $f$ is continuous,
the \emph{interval image} $f [0,t]$ is compact.
We view elements $f \in \cC ( [0,T] ; \R^d )$ as \emph{paths} indexed by time $[0,T]$, so that $f[0,t]$ is the section of the path up to time $t \in [0,T]$.

We need some notation and concepts from convex and integral geometry: we found \cite{gruber,sw} to be very useful.
 For a set $A \subseteq \R^d$, write 
 $\partial A$ for its boundary  
and $\Int (A) := A \setminus \partial A$ for its interior. 
For  $A \subseteq \R^d$
and a point $x \in \R^d$, set $\rho(x,A) := \inf_{y \in A} \rho(x,y)$,
with the usual convention that $\inf \emptyset = +\infty$.
 Write $\Sp_{d-1} := \{ e \in \R^d : \| e \| = 1 \}$
for the unit sphere in $\R^d$.

Let $\cK_d$ denote the collection of convex compact sets in $\R^d$, and 
$\cK^0_d := \{ A \in \cK_d : 0 \in A \}$
 those   that contain
 the origin. 
Given $A \in \cK_d$, for $r \geq 0$ set
\[ \pi_r ( A) := \{ x \in \R^d : \rho (x,A) \leq r \} ,\]
the \emph{parallel body} of $A$ at distance $r$.
 The \emph{support function} of $A \in \cK^0_d$ is $h_A$ defined by
 \[ h_A ( x ) := \sup_{y \in A} ( x \cdot y ) , ~~  x \in \R^d .\]
Note that $h_A: \R^d \to \RP$ determines $A$ via $A = \{ x : x \cdot e \leq h_A (e) \text{ for all } e \in \Sp_{d-1}\}$, and that,
for $A, B \in \cK^0_d$, we have $A \subseteq B$ if and only if $h_A(e) \leq h_B (e)$ for all $e \in \Sp_{d-1}$; see \cite[p.~56]{gruber}.
The Hausdorff metric on $\cK^0_d$
is defined for $A,B \in \cK^0_d$ by
\[ \rho_H ( A, B ) := \max \left\{ \sup_{x \in B} \rho(x,A) , \, \sup_{y \in A} \rho(y,B) \right\} .\]
Two equivalent descriptions of $\rho_H$ (see e.g.\ Proposition 6.3 of \cite{gruber}) are
\begin{align}
\label{eq:hausdorff_minkowski} \rho_H (A, B) & = \inf \left\{ r \geq 0 : A \subseteq \pi_r ( B ) \text{ and }  B \subseteq \pi_r ( A ) \right\};  \text{ and } \\
\label{eq:hausdorff_support} \rho_H (A,B) & = \sup_{e \in \Sp_{d-1} } \left| h_A (e) - h_B (e) \right| .
\end{align}
 
For the rest of this section we study some basic properties of the map from a continuous path to its convex hull.
Let $f \in \cC ([0,T] , \R^d)$. For any $t \in [0,T]$,   $f[0,t]$ is  compact, and hence
Carath\'eodory's theorem for convex hulls (see Corollary 3.1 of \cite[p.~44]{gruber})
shows that $\hull  f [0,t]  $ is also compact. So $\hull   f [0,t]   \in \cK_d$ is convex, bounded, and closed; in particular, it is a Borel set.

It mostly suffices to work with paths parametrized over 
$[0,1]$.
For $f \in \cC_d$, define
\[ H (f) := \hull   f [ 0,1 ]   .\]
The next result shows that the function $H : ( \cC^0_d , \rho_\infty ) \to ( \cK^0_d , \rho_H )$ is continuous.

\begin{lemma}
\label{lem:path-hull}
For any $f, g \in \cC^0_d$, we have $H(f), H(g) \in \cK^0_d$ and 
\begin{equation}
\label{eq:H-comparison}
  \rho_H ( H(f) , H(g) ) \leq \rho_\infty ( f,g).\end{equation}
\end{lemma}
 \begin{proof}
 Let $f, g \in \cC^0_d$. Then $H(f)$ and $H(g)$ are non-empty, as they 
 contain $f(0) = g(0)= 0$.
Consider  $x \in H(f)$. Since the convex hull of a set is the set of all convex combinations of points of the set
(see Lemma 3.1 of \cite[p.~42]{gruber}), 
there exist   $n \in \N$, 
weights 
$\lambda_1, \ldots, \lambda_n \geq 0$ with $\sum_{i=1}^n \lambda_i =1$,
and $t_1, \ldots, t_n \in [0,1]$ for which $x = \sum_{i=1}^n \lambda_i f (t_i)$.
 Then, taking $y = \sum_{i=1}^n \lambda_i g (t_i)$, we have that $y \in H(g)$ and, by the triangle inequality,
 \[ \rho (x,y)
\leq \sum_{i=1}^n \lambda_i \rho( f(t_i) , g(t_i) )
  \leq \rho_\infty (f,g) .\]
 Thus, writing $r = \rho_\infty (f,g)$,
 every $x \in H(f)$ has $x \in \pi_r ( H (g) )$,
 $H(g) \subseteq \pi_r ( H (f) )$. Thus, by \eqref{eq:hausdorff_minkowski}, we obtain \eqref{eq:H-comparison}.
 \end{proof}
  
 We end this section by showing that the map $t \mapsto \hull   f[0,t]  $ on $[0,T]$ is continuous if $f$ is continuous on $[0,T]$,
 so that the continuous trajectory $t \mapsto f(t)$ is accompanied by a continuous `trajectory' of 
 convex hulls. This observation was made
 by El Bachir \cite[pp.~16--17]{elbachir}; we take a different route based on the path-space result Lemma \ref{lem:path-hull}. First we need a lemma.
  
 \begin{lemma}
 \label{lem:path-stretch}
 Let $T >0$ and $f \in \cC ( [0,T] ; \R^d)$. Then the map defined for $t \in [0,T]$ by 
$t \mapsto  g_t$, where $g_t : [0,1] \to \R^d$ is given by $g_t (s) = f (t s)$, $s \in [0,1]$, is
a continuous function from $([0,T], \rho)$ to $( \cC_d, \rho_\infty )$.
 \end{lemma}
 \begin{proof}
 First we fix $t \in [0,T]$ and show that $s \mapsto g_t (s)$ is continuous, so that $g_t \in \cC_d$ as claimed.
 Since $f$ is continuous on the compact interval $[0,T]$, it is uniformly continuous,
 and admits  
 a monotone modulus of continuity  $\mu_f : \RP \to \RP$ such that $\rho ( f(s_1), f(s_2) ) \leq \mu_f ( \rho (s_1, s_2 ))$
for all $s_1, s_2 \in [0,T]$, and   $\mu_f ( r ) \downarrow 0$ as $r \downarrow 0$ (see e.g.~\cite[p.~57]{kallenberg}).
 Hence
 \[ \rho ( g_t (s_1) , g_t (s_2) ) = \rho ( f(ts_1) , f(ts_2) ) \leq \mu_f ( \rho (ts_1 , ts_2)) = \mu_f ( t \rho (s_1, s_2 ) ) ,\]
 which tends to $0$ as $\rho(s_1, s_2) \to 0$. 
 Hence $g_t \in \cC_d$.

It remains to show that $t \mapsto g_t$ is continuous. But on $\cC_d$,
\begin{align*} \rho_\infty ( g_{t_1}, g_{t_2} ) & = \sup_{s \in [0,1]} \rho ( f(t_1 s) , f(t_2 s) ) \\
& \leq \sup_{s \in [0,1]} \mu_f ( \rho (t_1 s, t_2 s) ) \\
& = \mu_f ( \rho ( t_1, t_2 ) ) ,\end{align*}
which tends to $0$ as $\rho (t_1, t_2) \to 0$, again using the uniform continuity of $f$.
\end{proof}

Here is the path continuity result for convex hulls of continuous paths; cf \cite[pp.~16--17]{elbachir}.

\begin{proposition}
\label{prop:point-hull}
Let $T >0$ and $f \in \cC^0 ( [0,T] ; \R^d)$. Then the map defined for $t \in [0,T]$ by 
$t \mapsto  \hull   f[0,t]  $ is
a continuous function from $([0,T], \rho)$ to $( \cK^0_d, \rho_H )$.
\end{proposition}
\begin{proof}
By Lemma \ref{lem:path-stretch}, $t \mapsto g_t$ is continuous, where $g_t(s) = f(ts)$, $s \in [0,1]$. Note that,
since $f(0)= 0$, $g_t \in \cC_d^0$.
But the sets $f [0,t]$ and $g_t [0,1]$ coincide, so $\hull   f [0,t]    = H (g_t) \in \cK_d^0$, and, by Lemma \ref{lem:path-hull}, $g_t \mapsto H(g_t)$ is continuous.
Thus $t \mapsto H(g_t)$ is the composition of two continuous functions, hence itself a continuous function.
\end{proof}

 \subsection{Functionals of planar convex hulls}
\label{sec:functionals}
 
Now, and for the rest of the paper, we return to $d=2$ to address our main questions of interest; parts of what follows carry over to general $d \geq 2$,
but we do not pursue that generality here.
We consider functionals $\cA: \cK_2 \to \RP$
and $\cL : \cK_2 \to \RP$ given by the area and the perimeter length of convex compact sets in the plane. Formally,
 we 
 define $\cA$ as Lebesgue measure on $\R^2$, and then 
 \begin{equation}
 \label{eq:L-def}
 \cL (A) := \lim_{r \downarrow 0} \left( \frac{\leb ( \pi_r (A)) - \leb (A)}{r} \right),
 \text{ for } A \in \cK_2  . \end{equation}
 The limit in \eqref{eq:L-def} exists by the \emph{Steiner formula} of integral geometry (see e.g.~\cite{sw}),
 which expresses $\leb ( \pi_r(A))$ as a quadratic polynomial in $r$ whose coefficients
 are given in terms of the \emph{intrinsic volumes} of $A$:
 \begin{equation}
 \label{eq:steiner}
\leb ( \pi_r(A)) = \leb (A) + r \cL (A) + \pi r^2 \1 \{ A \neq \emptyset \} .\end{equation}
 In particular,  with $\cH_{d}$ denoting $d$-dimensional Hausdorff measure,
 \[ \cL (A) = \begin{cases}\phantom{2} \cH_{1} ( \partial A ) & \text{if } \Int (A) \neq \emptyset , \\
2 \cH_{1} ( \partial A ) & \text{if } \Int (A)  = \emptyset . 
\end{cases} \]
  For $A \in \cK_2^0$,
 \emph{Cauchy's formula} states
\[ \cL (A) = \int_{\Sp_1} h_A ( e ) \ud e . \]
It follows from Cauchy's formula that $\cL$ is increasing in the sense that if $A, B \in \cK^0_2$ satisfy $A \subseteq B$, then $\cL (A) \leq \cL (B)$;
clearly the functional $\cA$ is also increasing. The next result shows that the functions $\cL$ and $\cA$ are both continuous from $(\cK^0_2 , \rho_H )$ to $( \RP , \rho)$.

\begin{lemma}
\label{lem:functional-continuity}
Suppose that $A, B \in \cK^0_2$. Then
 \begin{align}
\label{eq:L-comparison}
  \rho ( \cL(A) , \cL(B) ) & \leq 2 \pi \rho_H (A,B) ;\\
  \label{eq:A-comparison}
  \rho ( \cA(A) , \cA(B) ) & \leq \pi   \rho_H (A,B)^2 +  ( \cL(A) \vee \cL(B) ) \rho_H (A,B)  .
  \end{align}
\end{lemma}
\begin{proof}
First consider $\cL$. By Cauchy's formula and  the triangle inequality,
\[
\left| \cL (A) - \cL(B) \right|    = \left| \int_{\Sp_1} \left( h_{A} ( e ) - h_{B} ( e ) \right) \ud e \right| 
 \leq 2 \pi \sup_{e \in \Sp_1} \left| h_{A} ( e ) - h_{B} ( e ) \right|   ,\]
which with   \eqref{eq:hausdorff_support} gives  \eqref{eq:L-comparison}.

Now consider $\cA$. Set $r = \rho_H (A,B)$. Then, by \eqref{eq:hausdorff_minkowski}, $A \subseteq \pi_r (B)$.
Hence
\[ \cA (A) \leq \cA ( \pi_r (B) ) \leq \cA(B) + r \cL (B) + \pi r^2 ,\]
by \eqref{eq:steiner}. With the symmetric
 argument starting from $B \subseteq \pi_r (A)$, we get \eqref{eq:A-comparison}.
\end{proof}

\subsection{Brownian convex hulls as scaling limits}
\label{sec:Brownian-hulls}

The two different scalings outlined in Section \ref{sec:outline}, for the cases $\mu =0$ and $\mu \neq 0$,
lead to different scaling limits for the random walk. Both are associated with Brownian motion.

In the case $\mu =0$, the scaling limit is the usual planar Brownian motion, at least when $\Sigma = I$, the identity matrix.
Let $b:=( b(s) )_{s \in [0,1]}$ denote standard Brownian motion in $\R^2$, started at $b(0) = 0$.  
For convenience we may assume $b \in \cC_2^0$ (we can work on a probability space for which continuity holds for all sample points, rather than merely almost all).
For $t \in [0,1]$, let $h_t := \hull   b[0,t]   \in \cK_2^0$ denote the convex hull of the Brownian path up to time $t$.
By Proposition \ref{prop:point-hull}, $t \mapsto h_t$ is continuous. Much is known about the properties of $h_t$: see e.g.\ 
\cite{chm,elbachir,evans,klm}.
We also set
\[ \ell_t := \cL ( h_t ) , ~~\text{and}~~ a_t := \cA ( h_t ) ,\]
the perimeter length and area of the standard Brownian convex hull.
By Lemma \ref{lem:functional-continuity},
the processes $t \mapsto \ell_t$ and $t \mapsto a_t$  have continuous and non-decreasing sample paths.

We also need to work with the case of general covariances $\Sigma$;
to do so we introduce more notation and recall some facts about multivariate Gaussian random vectors.
For definiteness, we view vectors as Cartesian column vectors when required.
Since $\Sigma$ is positive semidefinite and symmetric, 
there is a (unique) positive semidefinite symmetric matrix square-root $\Sigma^{1/2}$
for which $\Sigma = (\Sigma^{1/2} )^2$.
The   map $x \mapsto \Sigma^{1/2} x$ associated with $\Sigma^{1/2}$ is a linear transformation on $\R^2$
with Jacobian $\det \Sigma^{1/2} = \sqrt{ \det \Sigma}$; 
hence  $\leb ( \Sigma^{1/2} A ) =  \leb (A)  \sqrt{ \det \Sigma }$
for any measurable $A \subseteq \R^2$. 

If $W \sim \cN (0, I)$, then $\Sigma^{1/2} W \sim \cN (0, \Sigma)$,
a bivariate normal distribution with mean $0$
and covariance $\Sigma$; the notation permits $\Sigma =0$,
in which case $\cN(0,0)$ stands for the degenerate
normal distribution with point mass at $0$. Similarly, given $b$ a standard Brownian motion on $\R^2$, the diffusion $\Sigma^{1/2} b$
is \emph{correlated} planar Brownian motion with covariance matrix $\Sigma$. 
 We write `$\Rightarrow$' to indicate weak convergence.

\begin{theorem}
\label{thm:limit-zero} 
 Suppose that   $\mu =0$.
Then, as $n \to \infty$,
   \[ n^{-1/2} \hull \{ S_0, S_1, \ldots, S_n \} \Rightarrow \Sigma^{1/2} h_1 , \]
   in the sense of weak convergence on $(\cK_2^0 , \rho_H )$.
  \end{theorem}
\begin{proof}
Donsker's theorem 
implies that $n^{-1/2} X_n \Rightarrow \Sigma^{1/2} b$
on $(\cC_2^0 , \rho_\infty )$. 
Now, the point set $X_n [0,1]$ is the union of the line segments
$\{  S_{k} +\theta (S_{k+1} - S_k)  : \theta \in [0,1] \}$
over $k=0,1,\ldots, n-1$. Since the convex hull is preserved under  affine transformations,   
\[ 
H ( n^{-1/2} X_n ) =
 n^{-1/2} H (X_n) = n^{-1/2} \hull \{ S_0, S_1, \ldots, S_n \}  .\]
By Lemma \ref{lem:path-hull}, $H$ is continuous, and so the  continuous mapping theorem 
(see e.g.~\cite[p.~76]{kallenberg}) implies that $n^{-1/2} \hull \{ S_0, S_1, \ldots, S_n \} \Rightarrow H ( \Sigma^{1/2} b )$ on $(\cK_2^0 , \rho_H )$.
Finally, invariance of the convex hull under affine transformations shows $H (\Sigma^{1/2} b ) = \Sigma^{1/2} H (b) = \Sigma^{1/2} h_1$.
\end{proof}

Theorem \ref{thm:limit-zero} together with the continuous mapping theorem and Lemma \ref{lem:functional-continuity}
implies the following distributional limit results in the case $\mu =0$. Here and subsequently `$\tod$' denotes
convergence in distribution for $\R$-valued random variables.

\begin{corollary}
\label{cor:zero-limits}
 Suppose that 
$\mu =0$. 
Then, as $n \to \infty$,
   \[ n^{-1/2} L_n \tod \cL ( \Sigma^{1/2} h_1 ) , ~~\text{and} ~~ n^{-1 } A_n \tod  \cA ( \Sigma^{1/2} h_1  ) = a_1 \sqrt{\det \Sigma} .\]
\end{corollary}

\begin{remark}
The distributional limits for $ n^{-1/2} L_n$ and $ n^{-1 } A_n$
in Corollary \ref{cor:zero-limits} are supported on $\RP$ and, as we will show in  Proposition \ref{prop:var_bounds} below, are non-degenerate if $\Sigma$ is positive definite;
hence they are \emph{non-Gaussian} excluding trivial cases.
\end{remark}

In the case $\mu \neq 0$,  the scaling limit 
can be viewed as a space-time trajectory of one-dimensional Brownian motion. Let $w:=( w(s) )_{s \in [0,1]}$ denote standard Brownian motion in $\R$, started at $w(0) = 0$;
similarly to above, we may take $w \in \cC_1^0$.
Define $\tilde b   \in \cC_2^0$ in Cartesian coordinates via
\[ \tilde b (s) =  (  s , w(s) ) , ~ \text{for } s \in [0,1]; \]
thus $\tilde b [0,1]$ is the space-time diagram of one-dimensional Brownian motion run for unit time.
For $t \in [0,1]$, let $\tilde h_t := \hull \tilde b [0,t] \in \cK_2^0$, and define $\tilde a_t := \cA ( \tilde h_t )$.
(Closely related to $\tilde h_t$ is the greatest \emph{convex minorant} of $w$ over $[0,t]$, which is
of interest in its own right, see e.g.~\cite{pr} and references therein.)
  
Suppose $\mu \neq 0$ and $\sperp \in (0,\infty)$. 
Given $\mu \in \R^2 \setminus \{ 0 \}$,
let $\hat \mu_\perp$ be the unit vector perpendicular   to  $\mu$ obtained by rotating $\hat \mu$ by $\pi/2$ anticlockwise.
For $n \in \N$, define $\psi^\mu_n : \R^2 \to \R^2$ by the image of $x \in \R^2$ in Cartesian components:
\[ \psi^\mu_n ( x ) = \left( \frac{x \cdot \hat \mu}{n \| \mu \|  } , \frac{x \cdot \hat \mu_\perp}{ \sqrt { n \sperp }  } \right) .\]
In words, $\psi^\mu_n$ rotates $\R^2$, mapping $\hat \mu$ to the unit vector in the horizontal direction,
and then scales space with a horizontal shrinking factor $\| \mu \| n$ and a vertical factor $ \sqrt { n \sperp } $.

 \begin{theorem}
  \label{thm:limit-drift}  
  Suppose that  $\mu \neq 0$, and $\sperp >0$.
 Then, as $n \to \infty$,
   \[ \psi^\mu_n (  \hull \{ S_0, S_1, \ldots, S_n \} ) \Rightarrow \tilde h_1, \]
   in the sense of weak convergence on $(\cK_2^0 , \rho_H )$.
  \end{theorem}
\begin{proof}
Observe that $\hat \mu \cdot S_n$ is a random walk on $\R$ with one-step mean drift $\hat \mu \cdot \mu = \| \mu \| \in (0,\infty)$,
while $\hat \mu_\perp \cdot S_n$ is a walk with mean drift $\hat \mu_\perp \cdot \mu = 0$
and increment variance
\[ \Exp \left[ ( \hat \mu_\perp \cdot Z  )^2   \right]
= \Exp \left[ ( \hat \mu_\perp \cdot ( Z - \mu)  )^2   \right]
=  \Exp [ \| Z - \mu \|^2 ] - \Exp [ (\hat \mu \cdot (Z - \mu) )^2 ] = \sigma^2 - \spara = \sperp .\]
According to the strong law of large numbers, for any $\eps>0$ there exists $N_\eps \in \N$ a.s.\ such that
$| m^{-1} \hat \mu \cdot S_m - \| \mu \| | < \eps$ for $m \geq N_\eps$.
Now we have that
\begin{align*} \sup_{N_\eps/n \leq t \leq 1 } \left| \frac{ \hat \mu \cdot S_{\lfloor nt \rfloor}}{n} - t \| \mu \| \right|
& \leq \sup_{N_\eps/n \leq t \leq 1 } \left( \frac{\lfloor nt \rfloor}{n} \right)
\left| \frac{ \hat \mu \cdot S_{\lfloor nt \rfloor}}{\lfloor nt \rfloor} - \| \mu \| \right| + \| \mu \| \sup_{0 \leq t \leq 1 } \left| \frac{\lfloor nt \rfloor}{n} - t \right| \\
& \leq  \sup_{N_\eps/n \leq t \leq 1}
\left| \frac{ \hat \mu \cdot  S_{\lfloor nt \rfloor}}{\lfloor nt \rfloor} - \| \mu \| \right| + \frac{ \| \mu \|}{n} \leq \eps +  \frac{ \| \mu \|}{n}.\end{align*}
On the other hand,
\[ \sup_{0 \leq t \leq N_\eps/n } \left| \frac{ \hat \mu \cdot S_{\lfloor nt \rfloor}}{n} - t \| \mu \| \right|
\leq \frac{1}{n} \max \{  \hat \mu \cdot  S_0,   \ldots,  \hat \mu \cdot S_{N_\eps} \}  + \frac{N_\eps \| \mu \|}{n} \to 0, \as ,\]
since $N_\eps < \infty$ a.s. Combining these last two displays and using the fact that $\eps>0$ was arbitrary, we see that $\sup_{0 \leq t \leq 1}  \left| n^{-1}   \hat \mu \cdot S_{\lfloor nt \rfloor} - t \| \mu \| \right| \to 0$, a.s.\ (the
functional version of the strong law). 
Similarly, $\sup_{0 \leq t \leq 1}  \left| n^{-1}   \hat \mu \cdot S_{\lfloor nt \rfloor +1} - t \| \mu \| \right| \to 0$, a.s.\ as well.
Since $X_n(t)$ interpolates $S_{\lfloor nt \rfloor}$ and  $S_{\lfloor nt \rfloor +1}$, it follows that
 $\sup_{0 \leq t \leq 1}  \left| n^{-1}   \hat \mu \cdot X_n(t)  - t \| \mu \| \right| \to 0$, a.s. In other words,
$(n \|\mu \|)^{-1} X_n \cdot \hat \mu$ converges a.s.\ to the identity function $t \mapsto t$ on $[0,1]$.

For the other component,  Donsker's theorem gives $( n \sperp)^{-1/2} X_n \cdot \hat \mu_\perp \Rightarrow w$ on $(\cC_1^0, \rho_\infty)$.
It follows that,   as $n \to \infty$,
  $\psi^\mu_n (  X_n ) \Rightarrow \tilde b$, 
   on $(\cC_2^0 , \rho_\infty )$. Hence by Lemma \ref{lem:path-hull} and since $\psi_n^\mu$ acts as an affine transformation on $\R^2$,
\[ \psi_n^\mu ( H ( X_n ) ) = H  ( \psi_n^\mu ( X_n  ) ) \Rightarrow H ( \tilde b ) ,\]
on $(\cK_2^0 , \rho_H )$, and the result follows.
\end{proof}

Theorem \ref{thm:limit-drift}  with the continuous mapping theorem, Lemma \ref{lem:functional-continuity},
and the fact that $\cA ( \psi_n^\mu ( A )) = n^{-3/2} \| \mu \|^{-1} ( \sperp )^{-1/2} \cA ( A )$
for measurable $A \subseteq \R^2$,
 implies
 the following distributional limit for $A_n$ in the case $\mu \neq 0$.

\begin{corollary}
\label{cor:A-limit-drift}
 Suppose that  $\mu \neq 0$, and $\sperp >0$.
Then 
   \[ n^{-3/2} A_n \tod \| \mu \| ( \sperp )^{1/2} \tilde a_1 , \text{ as }  n \to \infty .  \]
\end{corollary}

 \begin{remarks}
(i) Only the $\sperp >0$ case is non-trivial, since
$\sperp =0$ if and only if $Z$ is parallel to $\pm \mu$ a.s., in which case all the points
$S_0, \ldots, S_n$ are collinear and $A_n = 0$ a.s.\ for all $n$.\\
(ii)
The limit in Corollary \ref{cor:A-limit-drift} is non-negative and non-degenerate (see Proposition \ref{prop:var_bounds}
below) and hence non-Gaussian.
\end{remarks}

\section{Expectation and variance asymptotics}
\label{sec:asymptotics}

\subsection{Expectation asymptotics}

We start with asymptotics for $\Exp L_n$ and $\Exp A_n$ in the case $\mu =0$. These results,
Propositions \ref{prop:EL-zero} and \ref{prop:EA-zero},
are in part already contained in \cite{sw} and \cite{bnb} respectively; we give concise proofs here
  since several of the computations involved
will be useful later.
The first result, essentially given in \cite[p.\ 508]{sw}, is for $L_n$.

\begin{proposition} \label{prop:EL-zero}
 Suppose that $\mu=0$. Then, for $Y \sim \cN( 0, \Sigma)$,
\[ \lim_{n \to \infty} n^{-1/2}\Exp L_n = \Exp \cL (\Sigma^{1/2} h_1) = 4 \Exp \| Y \|. \]
\end{proposition}

  Cauchy's formula applied to the line segment from $0$ to $Y$ with Fubini's theorem implies $2 \Exp \| Y \| = \int_{\Sp_1} \Exp [ ( Y \cdot e )^+ ] \ud e$.
Here $Y \cdot e = e^\tra Y$ is univariate normal
with mean $0$ and variance $e^\tra \Sigma e = \|\Sigma^{1/2} e\|^2$,  so that $\Exp[ ( Y \cdot e)^+ ]$ is  
$\|\Sigma^{1/2} e\|$ times one half of the mean of the square-root of a $\chi_1^2$ random variable. Hence
$\Exp \| Y \| = ( 8 \pi)^{-1/2} \int_{\Sp_1} \|\Sigma^{1/2} e\| \ud e$, which 
in general may be expressed via a complete elliptic integral of the second kind
in terms of the ratio of the eigenvalues of $\Sigma$.
In the particular case $\Sigma = I$, $\Exp \| Y \| = \sqrt{\pi / 2}$ so
then Proposition \ref{prop:EL-zero} implies that
\[
\lim_{n \to \infty} n^{-1/2}\Exp L_n = \sqrt{8 \pi}, \]
 matching the formula 
$\Exp \ell_1 = \sqrt{8 \pi}$ of Letac and Tak\'acs \cite{letac,takacs}. 
We also note the bounds
\begin{equation}
\label{EL-bounds}
 \pi^{-1/2} \sqrt{ \trace \Sigma } \leq \Exp \| Y \| \leq \sqrt{ \trace \Sigma } ;\end{equation}
the upper bound here is from Jensen's inequality and the fact that $\Exp [ \| Y\|^2 ] = \trace \Sigma$.
The lower bound in \eqref{EL-bounds} follows from the inequality  
\[ \Exp \| Y \| \geq \sup_{e \in \Sp_1} \Exp | Y \cdot e |
=  \sqrt{ 2/\pi } \sup_{e \in\Sp_1}  ( \Var [ Y \cdot e ] )^{1/2} \]
together with the fact that 
\[ \sup_{e \in \Sp_1} \Var [ Y \cdot e ] 
=  \sup_{e \in \Sp_1} \|\Sigma^{1/2} e\|^2 
=   \| \Sigma^{1/2} \|^2_{\rm op} = \| \Sigma \|_{\rm op} = \lambda_\Sigma  \geq \frac{1}{2} \trace \Sigma  , 
\]
where $\| \blob  \|_{\rm op}$ is the matrix operator norm and $\lambda_\Sigma $ is the largest
eigenvalue  of $\Sigma$;
in statistical terminology, $\lambda_\Sigma$ is the variance of the first principal component associated with $Y$.

In what follows, we make repeated use of the following fact (see e.g.\ \cite[Lemma\ 4.11]{kallenberg}): if
random variables $\zeta ,\zeta_1, \zeta_2 , \ldots $ are such that $\zeta_n \to \zeta$ in distribution and the
$\zeta_n$ are uniformly integrable, then $\Exp \zeta_n \to \Exp \zeta$. 

\begin{proof}[Proof of Proposition \ref{prop:EL-zero}.]
The finite point-set case of Cauchy's formula gives
\begin{equation}
\label{eq:walk-cauchy}
L_n = \int_{\Sp_{1}} \max_{0 \leq k \leq n} ( S_k \cdot e)\ud e \leq 2 \pi \max_{0 \leq k \leq n} \| S_k \|.\end{equation}
Then by Lemma \ref{lem:walk_moments}(ii)  we have $\sup_n \Exp [ ( n^{-1/2} L_n )^{2} ] < \infty$.
Hence $n^{-1/2} L_n$ is uniformly integrable, so that Theorem \ref{thm:limit-zero} yields
$\lim_{n \to \infty} n^{-1/2}\Exp L_n = \Exp \cL ( \Sigma^{1/2} h_1 )$.

It remains to show that $\lim_{n \to \infty} n^{-1/2}\Exp L_n = 4 \Exp \| Y \|$. One can use Cauchy's formula to compute
$\Exp \cL ( \Sigma^{1/2} h_1 )$; instead we give a direct random walk argument, following \cite{sw}.
The central limit theorem for $S_n$ implies that  
 $n^{-1/2} \|S_n\| \to \|Y\|$ in distribution. Under the given conditions, $\Exp [ \|S_{n+1}\|^2 ] = \Exp [ \|S_{n}\|^2 ] + \Exp [ \|Z \|^2 ]$,
so that $\Exp [ \| S_n \|^2 ] = O(n)$. It follows that $n^{-1/2} \|S_n\|$ is uniformly integrable,
and hence  $\lim_{n \to \infty} n^{-1/2} \Exp \|S_n\| = \Exp \|Y\|$.
The result now follows from some standard analysis based on \eqref{eq:sw} and 
the fact that $\lim_{n \to \infty} n^{-1/2} \sum_{k=1}^n k^{-1/2} =2$.
\end{proof}
 
Now we move on to the area $A_n$. First we state some useful moments bounds.

\begin{lemma} 
\label{lem:A_moments}
Let $p \geq 1$. Suppose that $\Exp [ \| Z \|^{2p} ] < \infty$. 
\begin{itemize}
\item[(i)] We have $\Exp [ A_n^p ] = O ( n^{3p/2} )$.
\item[(ii)] Moreover, if $\mu =0$ we have $\Exp [ A_n^p ] = O (n^{p} )$.
\end{itemize}
\end{lemma}
\begin{proof}
First we prove (ii). Since   $\hull \{ S_0, \ldots, S_n\}$ 
is contained in the disk of radius $\max_{0 \leq m \leq n} \|S_m\|$ and centre $0$, we have
$A_n^p \leq \pi^p \max_{0 \leq m \leq n} \|S_m\|^{2p}$. Lemma \ref{lem:walk_moments}(ii) then yields part (ii).  
For part (i), it suffices to suppose $\mu \neq 0$. Then, bounding the convex hull by  a rectangle,
\begin{align*}
A_n
& \leq 
\left(\max_{0\leq m \leq n} S_m \cdot \hat \mu - \min_{0\leq m \leq n} S_m \cdot \hat \mu \right) \left(\max_{0\leq m \leq n} S_m \cdot \hat \mu_\perp - \min_{0\leq m \leq n} S_m \cdot \hat \mu_\perp \right)\\
& \leq 
4 \left(\max_{0\leq m \leq n} |S_m \cdot \hat \mu| \right) \left(\max_{0\leq m \leq n} |S_m \cdot \hat \mu_\perp| \right) .
\end{align*}
Hence, by the Cauchy--Schwarz inequality, we have
\[ \Exp [ A_n^p ] \leq 4^p \left( \Exp \left[ \max_{0\leq m \leq n} |S_m \cdot \hat \mu|^{2p} \right] \right)^{1/2} 
\left(\Exp \left[ \max_{0\leq m \leq n} |S_m \cdot \hat \mu_\perp|^{2p} \right] \right)^{1/2} .\]
Now an application of Lemma  \ref{lem:walk_moments}(i) and (iii) gives part (i).
\end{proof}

The asymptotics for $\Exp A_n$ in the case $\mu =0$ are given in the following result,
which is in part contained in \cite[p.\ 325]{bnb}.

\begin{proposition} 
\label{prop:EA-zero}
 Suppose that   $\mu=0$. 
Then, 
\[ \lim_{n \to \infty} n^{-1}\Exp A_n = \frac{\pi}{2} \sqrt{ \det \Sigma } . \]
\end{proposition}

Given Theorem \ref{thm:limit-zero}, one may also deduce the limit result in Proposition \ref{prop:EA-zero}
from the formula $\Exp a_1 = \frac{\pi}{2}$ of El Bachir \cite[p.~66]{elbachir} with a uniform integrability argument; however, the na\"ive approach
seems to require a slightly stronger moments assumption, such as \eqref{ass:moments} for some $p >2$ (cf Lemma \ref{lem:A_moments}).
The proof of Proposition \ref{prop:EA-zero}
is based on an  analogue for $\Exp A_n$ of the Spitzer--Widom formula, due to Barndorff-Nielsen and Baxter \cite{bnb}.
To state the formula, 
let $T(u,v)$ ($u,v \in \R^2$) be the area of a triangle with sides of $u,v$ and $u + v$. 
Note that for $\alpha, \beta >0$, $T(\alpha u, \beta v) = \alpha \beta T(u,v)$.
The formula of \cite{bnb} states
\begin{equation} 
\label{eq:bnb}
\Exp A_n = \sum_{k=2}^n \sum_{m=1}^{k-1} \frac{\Exp  T(S_m,S_k-S_m) }{m(k-m)} .
\end{equation}

\begin{proof}[Proof of Proposition \ref{prop:EA-zero}.]
First we show that, under the given conditions,
\begin{equation}
\label{eq:triangle_mean}
 \lim_{m \to \infty, \, k-m \to \infty} \frac{\Exp T(S_m, S_k-S_m)}{\sqrt{m(k-m)}} = \Exp T(Y_1,Y_2) , \end{equation}
where $Y_1$ and $Y_2$ are independent $\cN(0, \Sigma)$ random vectors.
Indeed, it follows from the central limit theorem in $\R^2$ 
and the continuity of $T$ that
\[ \frac{T(S_m, S_k-S_m)}{\sqrt{m(k-m)}} = T \left( \frac{S_m}{\sqrt{m}}, \frac{S_k-S_m}{\sqrt{k-m}} \right) \tod  T(Y_1,Y_2) , \]
  as $m \to \infty$ and $k-m \to \infty$. Moreover, $T(u, v) \leq \| u \| \| v\|$ so
\begin{align*}
\Exp \left[ \left( \frac{T(S_m, S_k-S_m)}{\sqrt{m(k-m)}} \right)^2 \right] 
& \leq \frac{\Exp [ \|S_m\|^2 \|S_k-S_m\|^2 ] }{m(k-m)} \\
& \leq \frac{\Exp [ \|S_m\|^2 ]}{m} \cdot\frac{\Exp [ \|S_k-S_m\|^2 ]}{k-m} ,
\end{align*}
which is uniformly bounded for $k \geq m+1 \geq 0$, by Lemma \ref{lem:walk_moments}.
It follows that $m^{-1/2}(k-m)^{-1/2}T(S_m,S_k - S_m)$ is uniformly integrable over $(m, k)$ with $m \geq 1$, $k \geq m+1$,
and the claim \eqref{eq:triangle_mean} follows.

With $\Sigma = (\Sigma^{1/2})^2$, we have that $(Y_1, Y_2)$ is equal in distribution to $(\Sigma^{1/2} W_1, \Sigma^{1/2} W_2)$
where $W_1$ and $W_2$ are independent $\cN (0, I)$ random vectors. Since $\Sigma^{1/2}$ acts as a linear transformation on $\R^2$
with Jacobian $\sqrt{ \det \Sigma}$,
\[ \Exp T(Y_1,Y_2) = \Exp T (\Sigma^{1/2} W_1, \Sigma^{1/2} W_2) = \sqrt{ \det \Sigma } \Exp T (W_1, W_2 ) .\]
 Here $\Exp T (W_1, W_2 ) = \frac{1}{2} \Exp [ \| W_1 \| \| W_2 \| \sin \Theta ]$,
where the minimum angle $\Theta$ between $W_1$ and $W_2$ is uniform on $[0, \pi]$, and $(\| W_1\|, \|W_2\|, \Theta)$
are independent. Hence  $\Exp T (W_1, W_2 ) = \frac{1}{2} ( \Exp \| W_1 \| )^2 ( \Exp \sin \Theta ) = \frac12$,
using the fact that $\Exp \sin \Theta = 2/\pi$ and $\| W_1 \|$ is the square-root of a $\chi_2^2$ random variable, so $\Exp \| W_1 \| = \sqrt{ \pi/2}$.

Thus from 
\eqref{eq:bnb}, \eqref{eq:triangle_mean}, and the computation $\Exp T(Y_1,Y_2) = \frac{1}{2} \sqrt{ \det \Sigma }$, we have
\begin{equation}
\label{eq:EA-1}
 \Exp A_n = \frac{1}{2} \sqrt{ \det \Sigma } \sum_{k=2}^n \sum_{m=1}^{k-1}  m^{-1/2} (k-m)^{-1/2} \left( 1 + \eps_{k,m} \right) ,\end{equation}
where, for any $\eps >0$, there exists $m_0 \in \N$ such that
 $| \eps_{k,m} | \leq \eps$ for all $m \geq m_0$ and $k-m \geq m_0$. 
Moreover,
\begin{equation}
\label{eq:pi-sum}
 \lim_{k \to \infty} \sum_{m=1}^{k-1} m^{-1/2} (k-m)^{-1/2} 
 =  
 \int_0^1 y^{-1/2} (1-y)^{-1/2} \ud y = \pi ,\end{equation}
so that the corresponding Ces\`aro limit also satisfies
\[ 
\lim_{n \to \infty} \frac{1}{n} \sum_{k=2}^n \sum_{m=1}^{k-1} m^{-1/2} (k-m)^{-1/2} = \pi . \]
With \eqref{eq:EA-1} it follows that, for any $\eps >0$,
\[ n^{-1} \Exp A_n \leq  \frac{\pi}{2} (1+ \eps) \sqrt{ \det \Sigma }  + O (n^{-1/2} ) ,\]
which gives $\limsup_{n \to \infty} n^{-1} \Exp A_n \leq \frac{\pi}{2} \sqrt{ \det \Sigma }$,
and a similar argument gives the corresponding
$\liminf$ result.
\end{proof}
 
Next we move on to the case $\mu \neq 0$. The following result on the asymptotics of $\Exp A_n$ in this
 case is, as far as we are aware, 
 new.
 
\begin{proposition} 
\label{prop:EA-drift}
 Suppose that \eqref{ass:moments} holds for some $p > 2$, $\mu \neq 0$,
and $\sperp >0$.
 Then 
\[ \lim_{n \to \infty} n^{-3/2} \Exp A_n = \| \mu \|  ( \sperp)^{1/2} \Exp \tilde a_1 = \frac{1}{3} \| \mu \|  \sqrt{2\pi \sperp} . \]
In particular,  $\Exp \tilde a_1 = \frac{1}{3} \sqrt{2 \pi}$.
\end{proposition}
\begin{proof}
Given  $\Exp [ \|Z_1\|^p ] < \infty$ for some $p >2$, Lemma \ref{lem:A_moments}(i) shows that
$\Exp [ A_n^{p/2} ] = O ( n^{3p/4} )$, so that 
$\Exp [ ( n^{-3/2} A_n )^{p/2} ]$
is uniformly bounded. Hence $ n^{-3/2} A_n$ is uniformly integrable,
so Corollary \ref{cor:A-limit-drift} implies that 
\begin{equation}
\label{eq:EA-scaling-drift}
 \lim_{n \to \infty} n^{-3/2} \Exp A_n = \| \mu \|  ( \sperp)^{1/2} \Exp \tilde a_1.
\end{equation}

In light of \eqref{eq:EA-scaling-drift},    it 
remains to identify $\Exp \tilde a_1= \frac{1}{3} \sqrt{2 \pi}$. It does not seem straightforward to work directly with the Brownian limit;
it turns out again to be
simpler to work with a suitable random walk.
We choose a walk that is particularly convenient for
computations.

Let $\xi \sim \cN (0,1)$ be a standard normal random variable,
and take $Z$ to be distributed as $Z = ( 1, \xi )$ in Cartesian coordinates. Then $S_n = ( n , \sum_{k=1}^n \xi_k )$ is
the space-time diagram of the symmetric random walk on $\R$ generated by i.i.d.\ copies $\xi_1, \xi_2, \ldots$ of $\xi$.

For $Z = (1, \xi)$, $\mu = (1,0)$ and $\sigma^2 = \sperp = \Exp [ \xi^2 ] = 1$. Thus 
by \eqref{eq:EA-scaling-drift}, to complete the proof of Proposition \ref{prop:EA-drift}
it suffices to show  that for this walk
$\lim_{n\to \infty} n^{-3/2} \Exp A_n = \frac{1}{3} \sqrt{2 \pi}$.
If $u, v \in \R^2$ have Cartesian components $u = (u_1, u_2)$ and $v = (v_1, v_2)$, then we
may write $T ( u, v) = \frac{1}{2} | u_1 v_2 - v_1 u_2 |$. Hence
\begin{align*}
T ( S_m, S_k - S_m ) & = \frac{1}{2} \left| ( k-m) \sum_{j=1}^m \xi_j - m \sum_{j=m+1}^k \xi_j \right| .\end{align*}
By properties of the normal distribution, the right-hand side of the last display has the same distribution as $ \frac{1}{2} | \xi  \sqrt{ k m (k-m)}  |$.
Hence
\[ \frac{\Exp T ( S_m, S_k - S_m )}{\sqrt{ m (k-m) } } = \frac{1}{2} \Exp |  \xi  \sqrt{k} | = \frac{1}{2} \sqrt{ 2 k / \pi } ,\]
using the fact that $| \xi |$ is distributed as the square-root of a $\chi_1^2$ random variable, so $\Exp | \xi | = \sqrt{ 2 / \pi }$.
Hence, by \eqref{eq:bnb}, this random walk enjoys the exact formula
\begin{align*} \Exp A_n & = \frac{1}{\sqrt{2\pi}} \sum_{k=2}^n \sum_{m=1}^{k-1} \frac{\sqrt{k}}{\sqrt{ m (k-m) }} . \end{align*}
Then from \eqref{eq:pi-sum} we obtain
$\Exp A_n \sim   \sqrt{\pi /2} \sum_{k=2}^n k^{1/2}$, which gives the result.
\end{proof}

\subsection{Variance asymptotics}

We are now able to give formally the results quoted in \eqref{eq:three_vars}, and to explain the constants that appear in the limits.
Indeed,
these are defined to be 
\begin{equation}
\label{eq:var_constants}
 u_0 ( \Sigma ) := \Var \cL ( \Sigma^{1/2} h_1 ) , ~~~
 v_0 := \Var a_1, ~~~ 
 v_+ := \Var \tilde a_1 .
\end{equation}

\begin{proposition}
\label{prop:var-limit-zero}
 Suppose that \eqref{ass:moments} holds for some $p > 2$, 
and $\mu =0$. Then
\[ \lim_{n \to \infty} n^{-1} \Var L_n = u_0 ( \Sigma ).\]
If, in addition, \eqref{ass:moments} holds for some $p > 4$, then
\[ \lim_{n \to \infty} n^{-2} \Var A_n = v_0 \det \Sigma.\]
\end{proposition}
\begin{proof}
From \eqref{eq:walk-cauchy} and Lemma \ref{lem:walk_moments}(ii), for $p > 2$ we have $\sup_n \Exp [ ( n^{-1} L_n ^2 )^{p/2} ] < \infty$.
Hence $n^{-1} L_n^2$ is uniformly integrable, and we deduce convergence of $n^{-1} \Var L_n$ in Corollary \ref{cor:zero-limits}.
Similarly, given  $\Exp [ \|Z_1\|^p ] < \infty$ for  $p >4$, Lemma \ref{lem:A_moments}(ii) shows that
$\Exp [ A_n^{2(p/4)} ] = O ( n^{p/2} )$, so that 
$\Exp [ ( n^{-2} A^2_n )^{p/4} ]$
is uniformly bounded. Hence $ n^{-2} A_n^2$ is uniformly integrable,
and we deduce convergence of $n^{-2} \Var A_n$ in Corollary \ref{cor:zero-limits}.
\end{proof}

For the case with drift, we have the following variance result.

\begin{proposition}
\label{prop:var-limit-drift}
Suppose that \eqref{ass:moments} holds for some $p > 4$ and $\mu \neq 0$.
Then
\[ \lim_{n \to \infty} n^{-3} \Var A_n = v_+ \| \mu \|^2 \sperp.\]
\end{proposition}
\begin{proof}
Given  $\Exp [ \|Z_1\|^p ] < \infty$ for some $p >4$, Lemma \ref{lem:A_moments}(i) shows that
$\Exp [ A_n^{2(p/4)} ] = O ( n^{3p/4} )$, so that 
$\Exp [ ( n^{-3} A^2_n )^{p/4} ]$
is uniformly bounded. Hence $ n^{-3} A_n^2$ is uniformly integrable,
so Corollary \ref{cor:A-limit-drift} yields the result.
\end{proof}

\subsection{Variance bounds}
 
The next result gives bounds on the quantities defined in \eqref{eq:var_constants}.

\begin{proposition}
\label{prop:var_bounds}
We have $u_0 (\Sigma) =0$ if and only if $\trace \Sigma =0$.
The following inequalities for the quantities defined at \eqref{eq:var_constants} hold.
\begin{alignat}{1}
  \frac{263}{1080} \pi^{-3/2}  \re^{-144/25} 
\trace \Sigma
&{} \leq  u_0 ( \Sigma)   \leq \frac{\pi^2}{2} \trace \Sigma ; \label{eq:u-bounds} \\
0 < \frac{4}{49} \left( \re^{- 7\pi^2 / 12} 
- 
\frac{1}{3} \re^{-21 \pi^2 / 4}
\right)^2  &{} \leq  v_0   \leq  16 (\log 2)^2 - \frac{\pi^2}{4}; \label{eq:v0-bounds} \\
0 < \frac{2}{225} \left( 
\re^{-25 \pi/9}
-\frac{1}{3} 
\re^{-25 \pi}
\right) &{} \leq  v_+   \leq 4 \log 2 - \frac{2 \pi}{9}. \label{eq:v1-bounds}
\end{alignat}
Finally, if $\Sigma = I$ we have the following sharper form of the lower bound in \eqref{eq:u-bounds}:
\[ \Var \ell_1 = u_0 (  I ) \geq \frac{2}{5} \left(1-\frac{8}{25\pi} \right) \re^{-25 \pi /16 } > 0 .\]
\end{proposition}

For the proof of this result, we rely on a few facts about one-dimensional Brownian motion,
including the bound  (see e.g.\ equation (2.1) of  \cite{jp}), valid for all $r>0$,
\begin{equation}
\label{eq:brown_norm}
\Pr \left[ \sup_{ 0 \leq s \leq 1} | w (s) | \leq r \right] \geq 
 \frac{4}{\pi} \left( \re^{-\pi^2/(8r^2)} - \frac{1}{3} \re^{-9\pi^2/(8r^2)} \right) .
\end{equation}
We let $\Phi$ denote the distribution function of a standard normal random variable; we will also need
the standard Gaussian tail bound (see e.g.~\cite[p.~12]{durrett})
\begin{equation}
\label{eq:gauss-bound}
1 - \Phi( x) = \frac{1}{\sqrt{2 \pi}} \int_x^\infty \re^{-y^2/2} \ud y \geq \frac{1}{x \sqrt{2 \pi}}   \left(1 - \frac{1}{x^2} \right) \re^{-x^2/2} , ~~ \text{for } x > 0 .
\end{equation}
We also note  that for $e \in \Sp_1$ the diffusion
$e \cdot (\Sigma^{1/2} b)$ is one-dimensional
 Brownian motion with
variance parameter $e^\tra \Sigma e$.

The idea behind the variance lower bounds is elementary. For a  random variable $X$ with mean $\Exp X$,
we have, for any $\theta \geq 0$, $\Var X = \Exp [ (X- \Exp X )^2 ] \geq \theta^2 \Pr [ | X - \Exp X | \geq \theta ]$.
If $\Exp X  \geq 0$,
taking $\theta = \alpha  \Exp X $ for $\alpha >0$, we obtain
\begin{equation}
\label{eq:var-bound}
 \Var X \geq \alpha^2 ( \Exp X )^2 \big( \Pr [ X \leq (1-\alpha)   \Exp X   ] + \Pr [ X \geq (1 + \alpha)   \Exp X   ] \big)
 ,\end{equation}
and  our lower bounds  use whichever of the
 latter two probabilities is most convenient.

\begin{proof}[Proof of Proposition \ref{prop:var_bounds}.]
We start with the upper bounds. Snyder and Steele's bound   \eqref{eq:ss}  with
the statement for $\Var L_n$ in
Proposition \ref{prop:var-limit-zero} gives the upper bound in \eqref{eq:u-bounds}.

Bounding $\tilde a_1$ by the area of a rectangle, we have
\begin{equation}
\label{eq:a1-upper}
\tilde a_1 \leq r_1 \leq 2 \sup_{0 \leq s \leq 1} | w (s) |, \as , \end{equation}
where $r_1 := \sup_{0 \leq s \leq 1} w (s) - \inf_{0 \leq s \leq 1} w (s)$. 
A result of Feller \cite{feller} states that $\Exp [ r_1^2 ] = 4 \log 2$.
So by the first inequality in \eqref{eq:a1-upper}, we have $\Exp [\tilde a_1^2] \leq 4 \log 2$, 
 and by Proposition \ref{prop:EA-drift}
we have $\Exp \tilde a_1 = \frac{1}{3} \sqrt{ 2 \pi}$; the upper bound in \eqref{eq:v1-bounds} follows.

Similarly, for any orthonormal basis $\{ e_1, e_2\}$ of $\R^2$, we bound $a_1$ by a rectangle
\[ a_1 \leq \left( \sup_{0 \leq s \leq 1} e_1 \cdot b(s) - \inf_{0 \leq s \leq 1 } e_1 \cdot b(s) \right) 
   \left( \sup_{0 \leq s \leq 1} e_2 \cdot b(s) - \inf_{0 \leq s \leq 1 } e_2 \cdot b(s) \right) ,\]
and the two (orthogonal) components are independent, so $\Exp [ a_1^2 ] \leq ( \Exp [ r_1^2 ] )^2 = 16 (\log 2)^2$,
which with the fact that $\Exp a_1 = \frac{\pi}{2}$ \cite{elbachir} gives the upper bound in \eqref{eq:v0-bounds}.

We now move on to the lower bounds.
Let $e_\Sigma \in \Sp_1$ denote an eigenvector of $\Sigma$ corresponding to the principal eigenvalue $\lambda_\Sigma$.
Then since $\Sigma^{1/2} h_1$ contains the line segment from $0$ to any (other) point in $\Sigma^{1/2} h_1$, we have from monotonicity of $\cL$ that
\[ \cL ( \Sigma^{1/2} h_1 )  \geq 2 \sup_{0 \leq s \leq 1 } \| \Sigma^{1/2} b(s) \| \geq 2 \sup_{0 \leq s \leq 1 } \left( e_\Sigma \cdot ( \Sigma^{1/2} b (s) ) \right) .\]
Here $e_\Sigma \cdot ( \Sigma^{1/2} b )$ has the same distribution as $\lambda_\Sigma^{1/2} w$. Hence, for $\alpha > 0$,
\begin{align*} \Pr \left[ \cL ( \Sigma^{1/2} h_1) \geq (1+ \alpha) \Exp \cL  ( \Sigma^{1/2} h_1) \right] & 
\geq 
\Pr \left[ \sup_{0 \leq s \leq 1} w(s) \geq \frac{1+\alpha}{2} \lambda_\Sigma^{-1/2} \Exp \cL  ( \Sigma^{1/2} h_1) \right] \\
& \geq 
\Pr \left[ \sup_{0 \leq s \leq 1} w(s) \geq 2 (1+\alpha) \sqrt{2}    \right] ,\end{align*}
using the fact that $\lambda_\Sigma \geq \frac{1}{2} \trace \Sigma$ and the upper bound
in \eqref{EL-bounds}. 
Applying \eqref{eq:var-bound} to $X = \cL  ( \Sigma^{1/2} h_1) \geq 0$ gives, for $\alpha >0$,
\begin{align*}
\Var \cL (\Sigma^{1/2} h_1 ) & \geq   \alpha^2  ( \Exp \cL  ( \Sigma^{1/2} h_1) )^2 \Pr \left[ \sup_{0 \leq s \leq 1} w(s) \geq 2 (1+\alpha) \sqrt{2}    \right] \\
& \geq \frac{32}{\pi}
\alpha^2 \left( \trace \Sigma \right) \left( 1 - \Phi ( 2 (1+\alpha) \sqrt{2}  ) \right)
,\end{align*}
using the lower bound in \eqref{EL-bounds} and the fact that
$\Pr [ \sup_{0 \leq s \leq 1 } w(s) \geq r ] = 2 \Pr [ w(1) \geq r ] = 2 ( 1 -\Phi (r))$ for $r >0$, which is
a consequence of the reflection principle. Numerical curve sketching suggests that $\alpha = 1/5$ is close to optimal; this choice of $\alpha$ gives,
using \eqref{eq:gauss-bound},
\[ \Var \cL (\Sigma^{1/2} h_1 ) \geq \frac{32}{25 \pi} \left( \trace \Sigma \right) \left( 1 - \Phi (  12 \sqrt{2} /5 ) \right)
\geq  \frac{263}{1080} \pi^{-3/2} \left( \trace \Sigma \right) \exp \left\{ -\frac{144}{25} \right\} ,\]
which is the lower bound in \eqref{eq:u-bounds}.
We get a sharper result when $\Sigma = I$ and $\cL ( h_1 ) = \ell_1$, since we know $\Exp \ell_1 = \sqrt{ 8 \pi}$ explicitly.
Then, similarly to above, we get 
\[ \Var \ell_1 \geq 8 \pi \alpha^2 \Pr 
 \left[ \sup_{0 \leq s \leq 1} w(s) \geq (1+\alpha) \sqrt{2 \pi}  \right] , \text{ for } \alpha >0, \]
which at $\alpha = 1/4$ yields the stated lower bound.

For areas,  tractable upper bounds for $a_1$ and $\tilde a_1$ are easier to come by than
lower bounds, and thus we obtain a lower bound on the variance by showing
the appropriate area has positive probability of being smaller than the corresponding mean.

Consider $a_1$; recall $\Exp a_1 = \frac{\pi}{2}$ \cite{elbachir}. Since, for any orthonormal basis $\{e_1, e_2\}$ of $\R^2$,
\[ a_1 \leq \pi \sup_{0 \leq s \leq 1} \| b (s) \|^2 \leq \pi \sup_{0 \leq s \leq 1 } |  e_1 \cdot b(s) |^2 +  \pi \sup_{0 \leq s \leq 1 } |  e_2\cdot b(s) |^2 ,\]
 using the fact that $e_1 \cdot b$ and $e_2 \cdot b$ are independent one-dimensional Brownian motions,
\[ \Pr [ a_1 \leq r ] \geq \Pr \left[   \sup_{0 \leq s \leq 1 } | w (s)   |^2 \leq \frac{r}{2\pi}  \right]^2  , ~ \text{for} ~ r >0 .\]
We apply \eqref{eq:var-bound} with $X = a_1$ and $\alpha \in (0,1)$, and set $r = (1-\alpha) \frac{\pi}{2}$ to obtain
\begin{align*}
\Var a_1 & \geq \alpha^2 \frac{\pi^2}{4}  \Pr \left[   \sup_{0 \leq s \leq 1 } |  w (s) | \leq \frac{\sqrt{1-\alpha}}{2}  \right]^2 \\
& \geq 4 \alpha^2 \left(  \exp \left\{-\frac{\pi^2}{2(1-\alpha) } \right\} 
- \frac{1}{3}  \exp \left\{-\frac{9 \pi^2}{2(1- \alpha) } \right\}  \right)^2 ,\end{align*}
by \eqref{eq:brown_norm}.
Taking $\alpha = 1/7$ is close to optimal, and gives the lower bound in \eqref{eq:v0-bounds}.

For $\tilde a_1$, we apply \eqref{eq:var-bound} with $X = \tilde a_1$ and $\alpha \in (0,1)$.
Using the fact
that $\Exp \tilde a_1 = \frac{1}{3} \sqrt{2 \pi}$ (from Proposition \ref{prop:EA-drift}) and the weaker of the two bounds in \eqref{eq:a1-upper},
we obtain
\begin{align*}
\Var \tilde a_1 & \geq \alpha^2 \frac{2 \pi}{9} \Pr \left[ \sup_{0 \leq s \leq 1} |  w (s) | \leq \frac{ (1-\alpha) \sqrt{2 \pi} }{6} \right]  \\
& \geq  \frac{8}{9} \alpha^2  \left( 
\exp \left\{-\frac{9\pi}{4 (1-\alpha)^2 } \right\} 
- \frac{1}{3}  \exp \left\{-\frac{81\pi}{4 (1-\alpha)^2 } \right\}   \right) ,\end{align*}
by \eqref{eq:brown_norm}.
Taking $\alpha = 1/10$ is close to optimal, and gives the lower bound in \eqref{eq:v1-bounds}.
\end{proof}

\begin{remarks}
(i) The main interest of the lower bounds in Proposition \ref{prop:var_bounds} is that they are \emph{positive};
they are certainly not sharp. The bounds can surely be improved, although
the authors have  been unable to improve any of them sufficiently  to warrant reporting the details here. We note just the following idea.
A lower bound for $\tilde a_1$ can be obtained by conditioning on $\theta := \sup \{ s \in [0,1] : w(s) =0\}$
and using the fact that the maximum of $w$ up to time $\theta$ is distributed as the maximum of a scaled Brownian bridge;
combining this with the previous argument improves the lower bound on $v_+$ to $2.09 \times 10^{-6}$.\\
(ii)
It would, of course, be of interest to evaluate any of $u_0$, $v_0$, or $v_+$ exactly. 
Although this looks hard, hope is provided by a remarkable computation by Goldman \cite{goldman} for the analogue of $u_0(I)= \Var \ell_1$ for the planar \emph{Brownian bridge}. Specifically, if
$b'_t$ is the standard Brownian bridge in $\R^2$ with $b'_0
= b'_1 = 0$, and $\ell'_1 = \cL ( \hull b' [0,1] )$ the perimeter length of its convex hull, 
 \cite[Th\'eor\`eme 7]{goldman} states that 
\[ \Var   \ell'_1 {} = {} 
\frac{\pi^2}{6} \left( 2 \pi \int_0^\pi \frac{ \sin \theta}{\theta} \ud \theta -   2 - 3 \pi \right) \approx 0.34755 . \]
\end{remarks}

\appendix
\section{Appendix: Random walk norms}
\label{sec:rw_norms}

\begin{lemma} 
\label{lem:walk_moments}
Let $p > 1$. Suppose that $\Exp[ \|Z_1\|^p ] < \infty$.
\begin{itemize}
\item[(i)] For any $e \in \Sp_1$ such that $e \cdot \mu = 0$, $\Exp  [ \max_{0\leq m\leq n} |S_m \cdot e|^p ]  = O(n^{1 \vee (p/2)})$. 
\item [(ii)] Moreover, if $\mu = 0$, then
$\Exp [  \max_{0\leq m\leq n} \| S_m  \|^p ]  = O(n^{1 \vee (p/2)})$.
\item[(iii)] On the other hand, if $\mu \neq 0$, then
$\Exp  [\max_{0\leq m\leq n} |S_m \cdot \hat \mu|^p ]= O ( n^p )$.  
\end{itemize}
\end{lemma}
\begin{proof}
Given that $\mu \cdot e =0$, $S_n \cdot e$ is a martingale, and hence, by convexity,
$| S_n \cdot e|$ is a non-negative submartingale. Then, for $p > 1$,   
\[ \Exp \left[ \max_{0\leq m \leq n} |S_m \cdot e|^p \right] \leq \left( \frac{p}{p-1} \right)^p \Exp \left[ | S_n \cdot e |^p \right]
= O ( n^{ 1 \vee (p/2) } ) ,\]
where the first inequality is Doob's $L^p$ inequality \cite[p.\ 505]{gut}
and the second is the Marcinkiewicz--Zygmund inequality \cite[p.\ 151]{gut}. This gives part (i).
 
Part (ii) follows from part (i): take $\{ e_1, e_2\}$ an orthonormal basis of $\R^2$ and apply (i) with each basis vector;
(ii) then follows from the triangle inequality $\max_{0\leq m\leq n} \| S_m  \|   \leq \max_{0\leq m\leq n} |S_m \cdot e_1 | + \max_{0\leq m\leq n} |S_m \cdot e_2 |$
together
with Minkowski's inequality.

Part (iii) follows from the fact that $\max_{0 \leq m \leq n} | S_m \cdot \hat \mu | \leq \sum_{k=1}^n | Z_k \cdot \hat \mu | \leq \sum_{k=1}^n \| Z_k \|$
and an application of Rosenthal's inequality \cite[p.\ 151]{gut} to the latter sum.
\end{proof}

\section*{Acknowledgements}

The authors are grateful to
Ben Hambly for discussions on scaling limits,
and to 
Ian Vernon for comments on multivariate normal distributions. 
This work was supported by the Engineering and Physical Sciences Research Council [grant number EP/J021784/1].


\begin{thebibliography}{99}

\bibitem{bnb} O.\ Barndorff-Nielsen and G.\ Baxter,
Combinatorial lemmas in higher dimensions,
\emph{Trans.\ Amer.\ Math.\ Soc.}\ {\bf 108} (1963) 313--325.

\bibitem{chm} M.\ Cranston, P.\ Hsu, and P.\ March, Smoothness of the convex hull of planar Brownian motion,
\emph{Ann.\ Probab.}\ {\bf 17} (1989) 144--150.

\bibitem{durrett} R.\ Durrett, \emph{Probability: Theory and Examples}, 4th ed., Cambridge University Press, Cambridge, 2010. 

\bibitem{elbachir} M.\ El Bachir, \emph{L'enveloppe convex du mouvement brownien}, Ph.D. thesis, Universit\'e Toulouse III---Paul Sabatier, 1983.

\bibitem{evans} S.N.\ Evans, On the {H}ausdorff dimension of {B}rownian cone points, \emph{Math.\ Proc.\ Camb.\ Philos.\ Soc.}\ {\bf 98} (1985) 343--353. 

\bibitem{feller} W.\ Feller, The asymptotic distribution of the range of sums of independent random variables,
\emph{Ann.\ Math.\ Statist.}\ {\bf 22} (1951) 427--432.

\bibitem{goldman} A.\ Goldman, Le spectre de certaines mosa\"iques poissoniennes du plan et l'enveloppe convex du pont brownien,
\emph{Probab.\ Theory Relat.\ Fields} {\bf 105} (1996) 57--83.

\bibitem{gruber} P.M.\ Gruber, \emph{Convex and Discrete Geometry}, Springer, Berlin, 2007.

\bibitem{gut} A.\ Gut, \emph{Probability: A Graduate Course}, Springer, Uppsala, 2005.

\bibitem{jp} N.C.\ Jain and W.E.\ Pruitt, The other law of the iterated logarithm,
\emph{Ann.\ Probab.}\ {\bf 3} (1975) 1046--1049.

\bibitem{kallenberg} O.\ Kallenberg, \emph{Foundations of Modern Probability}, 2nd ed., Springer, New York, 2002.

\bibitem{klm} J.\ Kampf, G.\ Last, and I.\ Molchanov, On the convex hull of symmetric stable processes,
\emph{Proc.\ Amer.\ Math.\ Soc.}\ {\bf 140} (2012) 2527--2535.

\bibitem{letac} G.\ Letac, Advanced problem 6230, \emph{Amer.\ Math.\ Monthly} {\bf 85} (1978) 686.

\bibitem{levybook} P.\ L\'evy, \emph{Processus Stochastiques et Mouvement Brownien},
Gauthier-Villars, Paris, 1948.

\bibitem{mcr} S.N.\ Majumdar, A.\ Comtet, and J.\ Randon-Furling,
Random convex hulls and extreme value statistics,
{\em J.\ Stat.\ Phys.}\ {\bf 138} (2010) 955--1009.

\bibitem{pr} J.\ Pitman and N.\ Ross, The greatest convex minorant of Brownian motion, meander, and bridge,
{\em Probab.\ Theory Relat.\ Fields} {\bf 153} (2012) 771--807.

\bibitem{schneiderweil} R.\ Schneider and W.\ Weil, Classical stochastic geometry, pp.~1--42 in \emph{New Perspectives
in Stochastic Geometry}, W.S.~Kendall \& I.~Molchanov (eds.), OUP, 2010.

\bibitem{ss} T.L.\ Snyder and J.M.\ Steele, Convex hulls of random walks,
{\em Proc.\ Amer.\ Math.\ Soc.}\ {\bf 117} (1993) 1165--1173.

\bibitem{sw} F.\ Spitzer and H.\ Widom, The circumference of a convex polygon,
{\em Proc.\ Amer.\ Math.\ Soc.}\ {\bf 12} (1961) 506--509.

\bibitem{takacs} L.\ Tak\'acs, Expected perimeter length, \emph{Amer.\ Math.\ Monthly} {\bf 87} (1980) 142.

\bibitem{wx} A.R.\ Wade and C.\ Xu, Convex hulls of planar random walks with drift, {\em Proc.\ Amer.\ Math.\ Soc.} To appear.

\end{thebibliography}
\end{document}